\documentclass[oneside, 12pt, english]{amsart}
\usepackage[T1]{fontenc}
\usepackage{mathtools}
\usepackage{amsbsy}
\usepackage{amssymb}
\usepackage{stackrel}
\usepackage{esint}
\usepackage[numbers]{natbib}
\usepackage{babel}
\newtheorem{theorem}{Theorem} 
\newtheorem{lemma}[theorem]{Lemma}
\newtheorem{proposition}[theorem]{Proposition}

\theoremstyle{definition}

\theoremstyle{remark}
\newtheorem{remark}{Remark} 

\numberwithin{equation}{section}

\begin{document}

\title[Common Limit of the Linear Statistics of Random Zeros]
{The Common Limit of the Linear Statistics of Zeros of Random
Polynomials and Their Derivatives}

\author{I-Shing Hu}
\address{Department of Mathematics, National Taiwan University, Taipei 10617,
Taiwan, R.O.C.}
\email{r02221015@ntu.edu.tw}

\author{Chih-Chung Chang}
\address{Department of Mathematics, National Taiwan University, Taipei 10617,
Taiwan, R.O.C.}
\email{ccchang@math.ntu.edu.tw}

\subjclass[2010] {Primary: 30C15, 60B10; secondary: 60B20, 60G57, 60F05, 60F15, 60F25}



\begin{abstract}
Let $ p_n(x) $ be a random polynomial of degree $n$ and
$\{Z^{(n)}_j\}_{j=1}^n$ and $\{X^{n, k}_j\}_{j=1}^{n-k}, k<n$,
be the zeros of $p_n$ and $p_n^{(k)}$, the $k$th derivative
of $p_n$, respectively.
We show that if the linear statistics 
$\displaystyle{ 
\frac {1}{a_n}   \left[ f\left( \frac {Z^{(n)}_1}{b_n} \right) 
+ \cdots + f \left(\frac {Z^{(n)}_n}{b_n} \right) \right]}$
associated with $\{Z^{(n)}_j\}$
has a limit as $n\to\infty$ at some mode of convergence,
the linear statistics associated with $\{X^{n, k}_j\}$
converges to the same limit at the same mode.
Similar statement also holds for the centered linear statistics 
associated with the zeros of $p_n$ and $p_n^{(k)}$,
provided the zeros $\{Z^{(n)}_j\}$ and the
sequences $\{a_n\}$ and $\{b_n\}$ of positive numbers
satisfy some mild conditions.
\end{abstract}

\maketitle

\section{Introduction} 

Fix a probability space $(\Omega, \mathcal F, \bf P)$
and let $\{p_n(z)\}_{n=1}^\infty$
be a sequence of random polynomials such that $\text{deg } p_n=n$.
We observe that the randomness of  polynomials can be introduced
in several different ways.

\begin{description}
\item [Type 1]
Given a triangular array of random zeros $\{Z^{(n)}_j\}_{j=1}^n, n=1, 2, \ldots$, let 
\begin{equation} \label{p1}
p_n(z)=(z-Z_1^{(n)})\cdots (z-Z_n^{(n)}).
\end{equation}
Here the coefficient of $z^n$ is set to be 1 for simplicity.
\item [Type 2]
Given a triangular array of random coefficients $\{a^{(n)}_j\}_{j=0}^n, n=1, 2, \ldots$, let 
\begin{equation} \label{p2}
p_n(z)=a^{(n)}_n z^n+\cdots +a^{(n)}_1 z+ a^{(n)}_0.
\end{equation}
\item [Type 3]
Given a sequence of random matrices $\{A^{(n)}\}_{n=1}^\infty$, let
\begin{equation} \label{p3}
p_n(z)=  { \rm{det} }(z I-A^{(n)}),
\end{equation}
the characteristic polynomial of $A^{(n)}$. 
Here $I$, the identity matrix, and $A^{(n)}$ are square matrices of size $n$.
\end{description}

\vskip .4cm
No matter how $p_n$ is constructed, 
we always denote the zeros of $p_n$ by $Z^{(n)}_1, \ldots , Z^{(n)}_n$.
Next, for each positive integer $k < n$,   
let $p_n^{(k)}$ be the $k$th derivative
of $p_n$, and $X^{n, k}_j, j=1, 2, \ldots , n-k$,
be the zeros of $p_n^{(k)}$.
In particular, when $k=1$, $X^{n, 1}_1, \cdots, X^{n, 1}_{n-1}$ are 
called the critical points of $p_n$.

The relation between the zeros and the critical points of
polynomials has been much studied. For example,
the Gauss-Lucas theorem asserts that all critical points
of a non-constant polynomial $f$ lie inside the closed convex hull 
formed by the zeros of $f$. It follows by induction that
the zeros of $f^{(k)}, k< \text{deg } f$, also lie inside the same closed convex hull.
More refinements of Gauss-Lucas theorem
can be found in \citep{or} and the references therein.
A recent paper \citep{a} also discussed some related results and examples.

On the other hand, R. Pemantle and I. Rivin initiated a
probabilistic study on the limit of the critical points of
random polynomials of Type 1.
Consider the following probability measures:
\begin{equation}\label{mu}
\mu_n=\frac{1}{n} \sum_{j=1}^{n} \delta_{Z^{(n)}_j}, \qquad\text{and}\qquad
\mu^{(k)}_n=\frac{1}{n-k} \sum_{j=1}^{n-k} \delta_{X^{n, k}_j}, 1\le k <n,
\end{equation}
where $\delta_z$ is the Dirac measure concentrated on $z$.
$\mu_n$ and $\mu^{(k)}_n$ are   
the empirical measures associated with the zeros $\{Z^{(n)}_j\}_{j=1}^{n}$
and $\{X^{n, k}_j\}_{j=1}^{n-k}$, respectively.
R. Pemantle and I. Rivin (\citep{pr}) showed that,
if $Z^{(n)}_j = Z_j$ and $\{Z_j\}_{j=1}^\infty$ is a sequence of
independent and identically distributed (i.i.d.) random variables 
governed by a common law $\nu$, then  
$\mu^{(1)}_n \stackrel {w}{\to} \nu$  
almost surely (a.s.) as $n\to \infty$ provided $\nu$ 
satisfies certain energy condition.
In this paper $\stackrel {w}{\to}$ means
``{\it converges weakly}'' or ``{\it converges in distribution}''.
In the same i.i.d. setting without any further assumption on 
the probability law $\nu$,
Z. Kabluchko (\citep{k}) proved in great generality that 
$\mu^{(1)}_n \stackrel {w}{\to} \nu$ in probability as $n\to\infty$.
For the case of higher order derivatives, in the i.i.d. setting, 
if the probability measure $\nu$ 
is supported  on the unit circle in $\mathbb C$, P. L. Cheung
et. al. (\citep{cnty}) showed that $\mu^{(k)}_n \stackrel{w}{\to} \nu$ a.s. 
as $n\to\infty$. Similar results for the zeros of the generalized derivatives 
of polynomials are also obtained in \citep{cnty}.

To state a result of Type 2
polynomials, recall that a polynomial is called a {\it Kac polynomial} (\citep{kz}) 
if it has the form $\displaystyle{\sum_{j=0}^n \xi_j z^j}$,
where $\{\xi_j\}_{j=0}^\infty$ is a sequence of non-degenerate 
i.i.d. random variables.
Furthermore, given a sequence of deterministic complex numbers
$\{w_j\}_{j=0}^\infty$, a polynomial of the form
$\displaystyle{\sum_{j=0}^n \xi_j w_j z^j}$ is called
a {\it Littlewood-Offord random polynomial}.
Clearly, any $k$th derivative of  
a Kac polynomial is a Littlewood-Offord random polynomial.
Z. Kabluchko and D. Zaporozhets (\citep{kz} and 
Theorem 14 of \citep{or}) 
proved that both sequences of the empirical measures
$\{\mu_n\}$ and $\{\mu^{(k)}_n\}$ 
converge weakly
to the uniform distribution on the unit circle of $\mathbb C$
centered at the origin in probability as $n\to\infty$,
provided that ${\bf E} [ \log ( 1 + | \xi_0 | ) ] <\infty$.
See \citep{kz} for the explicit statements
of the theorems and examples.

Under various settings and assumptions, the eigenvalue statistics of 
random matrices/sample covariance matrices exhibits
various interesting and important limit behavious, for example the circle law, semicircle law,
Mar\v{c}enko-Pastur law, central limit theorem, large deviations, and so on.
See \citep{agz} and \citep{bs} for a systematic introduction.
To name a result related to our work, 
consider a sequence of random Hermitian matrices 
$\{A^{(n)}\}_{n=1}^\infty$ with
$p_n$ its characteristic polynomial.
S. O'Rourke (\citep{or}) 
showed that the L\'{e}vy distance between
$\mu_n$ and $\mu^{(1)}_n$ tends to zero almost surely as $n\to\infty$.
This observation 
implies that a.s. $\{\mu^{(1)}_n\}$ converges weakly to the same semicircle law
as $\{\mu_n\}$ does. 
Such phenomenon that $\{\mu^{(1)}_n\}$ 
converges weakly to the same law as that of $\{\mu_n\}$ 
(at some mode of convergence) was further
demonstrated for several compact classical matrix groups 
by S. O'Rourke in the same paper.  
Check \citep{or} (theorem 6, corollary 7, theorem 9,
and remark 10) for explicit statements and references therein. 
Here we merely point out that all the limit laws of the eigenvalue statistics of the 
matrix models considered are compactly supported in $\mathbb C$. 

In view of the fact that all the results concerning the relation between
$\{\mu_n\}$ and $\{\mu^{(k)}_n\}$ reviewed above 
are of the type of law of large numbers, it is natural to ask how about other 
types of limit theorem? The goal of this paper is to show that
if certain limit property, for example law of large numbers, central limit
theorem, law of iterated logarithm, and so on,
holds for the linear statistics of $\{Z^{(n)}_j\}_{j=1}^n$
(see below for the precise statement), then the same limit property
passes to that of $\{X^{n, k}_j\}^{n-k}_{j=1}$ for any $k$,
provided the zeros $\{Z^{(n)}_j\}_{j=1}^n$ satisfy some mild
conditions which we now state.

Denote by $\Im z$ the imaginary part of a complex number $z$.

{\bf A1}.
There exists a non-negative constant $C_0 \ge 0$ independent of $n$ such that 
\begin{equation} \label{A1}
\sup_{n\in \mathbb N} \max_{1 \le j \le n} | \Im Z^{(n)}_j | \le C_0 \quad \text{ a.s.}
\end{equation}
That is, the imaginary parts of  $\{Z^{(n)}_j\}$ are
uniformly bounded with probability one. Since every zero of $p_n^{(k)}$ lies inside the closed convex hull 
of the zeros of $p_n^{(k-1)}$ by Gauss-Lucas theorem, we know by induction
that 
\begin{equation} \label{A1.1}
\sup_{n\in \mathbb N} \max_{1\le k < n} \max_{1 \le j \le n-k} 
| \Im X^{n,k}_j | \le C_0 \quad \text{ a.s.}
\end{equation}
When the zeros are real numbers, we put $C_0=0$.

Recall that $\mu_n \stackrel {\text{w} }{\to} \nu$ is equivalent
to $\displaystyle{\lim_{n\to\infty} \int f \,d\mu_n=
\lim_{n\to\infty} \frac 1n\sum_{j=1}^n f (Z^{(n)}_j ) = \int f\,d\nu}$
for each bounded continuous function $f$.
It is therefore quite common to study such sums (namely, linear statistics)
for various categories of test functions.
\citep{lp09}, \citep{lp11}, \citep{bwz}, \citep{jss}, and \citep{lrs}
are a few examples. In particular, the issue of 
regularity conditions for the test functions is discussed in \citep{ko} and \citep{sw}. 
In this paper we restrict ourselves to regular test functions
$f$ such that $f$ and $|f'|$ are bounded, $\hat f$ exists and $f(x) = \int \hat{f}(t) e^{ i t x }\,dt$, and
\begin{equation}\label{f-bdd}
\int_{\mathbb R} | \hat{f}(t) | e^{3 C_0 |t|}\,dt < \infty,
\end{equation}
where $i=\sqrt{-1}, \hat{f}(t)=\frac{1}{2\pi}\int f(u)e^{-i t u}\,du$ is the Fourier transform of $f$, and 
$C_0$ is the same absolute constant appeared in (\ref{A1}).

To adapt to the different scalings in various limit theorems, 
we consider different sequences of positive numbers 
for different linear statistics to be defined later.
Below we list three groups of assumptions to be used in the three
main theorems of this paper, respectively.

{\bf A2}.
There exists a sequence of positive numbers $\{a_n\}_{n=1}^\infty$ 
such that
\begin{eqnarray} 
&& \lim_{n\to\infty} a_n = \infty, \label{an1}\\ 
&& \lim_{n\to\infty} \frac{ n\,| a_{n-k} - a_{n-k-1} | }
{ a_{n-k} a_{n-k-1} }= 0, \; \text{for each fixed} \; k < n-1, \label{an2}\\
&& \lim_{n\to\infty} \sup_{1 \le j \le n} \left\{\frac {|Z^{(n)}_j |}{a_{n-k}}\right\}
=0 \;\;\text{ a.s.} \;\text{for each fixed} \; k<n.  \label{an3} 
\end{eqnarray}

{\bf A3}.
There exists a sequence of positive numbers 
$\{b_n\}_{n=1}^\infty$ such that
\begin{eqnarray} 
&& \lim_{n\to\infty} b_n = \infty, \label{bn1}\\
&& \lim_{n\to\infty} \frac{ | b_{n-k} - b_{n-k-1} |}
{ b_{n-k} b_{n-k-1} }\left[ \sum_{j=1}^n |Z^{(n)}_j| \right]
=0 \; \;\text{ a.s.} \; \text{for each fixed} \; k < n-1, \label{bn2}\\
&& \lim_{n\to\infty} \sup_{1 \le j \le n} \left\{\frac {|Z^{(n)}_j |}{b_{n-k}}\right\}
=0 \;\;\text{ a.s.} \;\text{for each fixed} \; k<n.  \label{bn3} 
\end{eqnarray}

{\bf A4}.
There exist two sequences of positive numbers 
$\{a_n\}_{n=1}^\infty$ and $\{b_n\}_{n=1}^\infty$ such that
\begin{eqnarray} 
&& \lim_{n\to\infty} a_n = \lim_{n\to\infty} b_n = \infty, \label{anbn1}\\
&& \lim_{n\to\infty} \frac{ | b_{n-k} - b_{n-k-1} |}
{ a_{n-k-1} b_{n-k} b_{n-k-1} }\left[ \sum_{j=1}^n |Z^{(n)}_j| \right]
=0 \; \;\text{ a.s.} \; \text{for each fixed} \; k < n-1, \label{anbn2}\\
&& \lim_{n\to\infty} \sup_{1 \le j \le n} \left\{\frac {|Z^{(n)}_j |}{a_{n-k} b_{n-k}}\right\}
=0 \;\;\text{ a.s.} \;\text{for each fixed} \; k<n.  \label{anbn3} 
\end{eqnarray}

Now we are ready to introduce the key objects that we want to study.
Consider the following three linear statistics associated with $\{Z^{(n)}_j\}$
\begin{eqnarray*}
L_{n, 1} (f) &=& \frac {1}{a_n}   \left[ f\left( Z^{(n)}_1 \right) 
+ \cdots + f \left( Z^{(n)}_n \right) \right], \\
L_{n, 2} (f) &=& f\left( \frac {Z^{(n)}_1}{b_n} \right) 
+ \cdots + f \left(\frac {Z^{(n)}_n}{b_n} \right), \\
L_{n, 3} (f) &=& \frac {1}{a_n}   \left[ f\left( \frac {Z^{(n)}_1}{b_n} \right) 
+ \cdots + f \left(\frac {Z^{(n)}_n}{b_n} \right) \right],
\end{eqnarray*}
and the three linear statistics associated with $\{X^{n, k}_j\}$
\begin{eqnarray*}
L^{(k)}_{n, 1} (f) &=& \frac {1}{a_{n-k}}   \left[ f\left( X^{n, k}_1 \right) 
+ \cdots + f \left( X^{n, k}_{n-k} \right) \right], \\
L^{(k)}_{n, 2} (f) &=& f\left( \frac {X^{n, k}_1}{b_{n-k}} \right) 
+ \cdots + f \left(\frac {X^{n, k}_{n-k}}{b_{n-k}} \right), \\
L^{(k)}_{n, 3} (f) &=& \frac {1}{a_{n-k}}   \left[ f\left( \frac {X^{n, k}_1}{b_{n-k}} \right) 
+ \cdots + f \left(\frac {X^{n, k}_{n-k}}{b_{n-k}} \right) \right].
\end{eqnarray*}
It is also necessary to consider the centered (mean zero) linear statistics
\begin{equation*}
{\bar L_{n, \ell}} (f)  =  L_{n, \ell} (f) - E [ L_{n, \ell} (f) ], 
\quad {\bar L^{(k)}_{n, \ell}} (f)  =  L^{(k)}_{n, \ell} (f) - E [ L^{(k)}_{n, \ell} (f) ],
\; \ell=1, 2, 3.
\end{equation*}

For example, $L_{n, 1} (f)$ with $a_n=n$ and 
${\bar L}_{n, 1} (f)$ with $a_n=\sqrt{n}$ 
play the typical roles in law of large numbers and 
central limit theorem, respectively.
In the random matrix models, one 
studies $L_{n, 3}$ with  $a_n=n, b_n=\sqrt{n}$ for law of large numbers results.
In these cases {\bf A2} and {\bf A4} are valid obviously.

\begin{theorem} \label{l1}
Let the random zeros $\{Z^{(n)}_j\}_{j=1}^n$ and the 
sequence $\{a_n\}$ 
of positive numbers be given as above such that they satisfy
the assumptions {\rm \bf A1} and {\rm \bf A2}. 
If the linear statistics 
$\displaystyle{L_{n, 1} (f)}$
has a limit as $n\to\infty$ at some mode of convergence,
then, for each $k<n$,
the linear statistics $L_{n, 1}^{(k)} (f)$ 
converges to the same limit at the same mode of convergence.
Similar statement holds for 
${\bar L}_{n, 1} (f)$ and ${\bar L}^{(k)}_{n, 1} (f)$.
\end{theorem}

\begin{theorem} \label{l2}
Let the random zeros $\{Z^{(n)}_j\}_{j=1}^n$ and the 
sequence $\{b_n\}$ 
of positive numbers be given as above such that they satisfy
the assumptions {\rm \bf A1} and {\rm \bf A3}. 
If the linear statistics 
${\bar L}_{n, 2} (f)$
converges weakly to some probability law $\nu=\nu_f$ as $n\to\infty$,
then ${\bar L}_{n, 2}^{(k)} (f) \stackrel {w}{\to} \nu_f$
for each fixed $k<n$.
\end{theorem}

\begin{theorem} \label{l3}
Let the random zeros $\{Z^{(n)}_j\}_{j=1}^n$ and the 
sequences $\{a_n\}$ and $\{b_n\}$
of positive numbers be given as above such that they satisfy
the assumptions {\bf A1} and {\bf A4}. 
If the linear statistics 
$\displaystyle{L_{n, 3} (f)}$
has a limit as $n\to\infty$ at some mode of convergence,
then, for each $k<n$,
the linear statistics $L_{n, 3}^{(k)} (f)$ 
converges to the same limit at the same mode of convergence.
Similar statement holds for 
${\bar L}_{n, 3} (f)$ and ${\bar L}^{(k)}_{n, 3} (f)$.
\end{theorem}

\begin{remark}
Consider the {\bf Type 1} random polynomials with 
$Z^{(n)}_j=Z_j$ and $\{Z_j\}$ being an i.i.d. sequence.
In this case the uniform condition (\ref{an3}) in {\bf A2} can be weakened to
$$
\lim_{n\to\infty} \frac { |Z_1| + \cdots + |Z_n|} {a_{n-k} n}=0
\;\;\text{ a.s.} \; \text{for each fixed} \;k<n.  
$$
\end{remark}
\begin{remark}
Again consider the {\bf Type 1} random polynomials with 
$Z^{(n)}_j=Z_j$ and $\{Z_j\}$ being an i.i.d. sequence.
This is the setting studied in \citep{pr}, \citep{k}, and \citep{cnty}.
Note that our Theorem 1 establishes, in addition to law of large numbers result,
also central limit theorem, law of iterated logarithm, and so on,
for the linear statistics of $\{X^{n, k}_j\}$.
However, Theorem 2 is not as interesting since it does not include
the central limit theorem of random matrix models. 
This is because the size of the zeros $Z^{(n)}_j$ 
is of order $\sqrt{n}=b_n$, and therefore the conditions (\ref{bn2}) and (\ref{bn3})
would not hold. Reasonable condition(s) should involve
the centered quantities.
\end{remark}

In Section 2 we establish
a comparison identity. It is elementary, and yet
crucial to our results. In Section 3 we prove
three theorems and 
make some final remarks.

\section{A comparison identity}

First we state (with some modifications) a theorem of 
Cheung and Ng (\citep{cn}) which is the starting point of our argument.

\begin{proposition} [Theorem 1.1 of \citep{cn}]
The set of all critical points of $p_{n}(z) = \prod_{k=1}^n (z-z_{k}), n\ge 2$,
is the same as the set of all eigenvalues of the $(n-1)\times(n-1)$ matrix $M_{n-1}$\,:
\begin{equation} \label{m}
M_{n-1} = D_{n-1} + \frac 1n \left( z_1 I_{n-1} - D_{n-1}\right) J_{n-1},
\end{equation}
where 
\begin{equation} \label{d}
D_{n-1} = \left(\begin{array}{cccc}
z_2 & 0 & \cdots &  0\\
0 & z_3 & \cdots & 0\\
\vdots & \vdots & \ddots & \vdots\\
0 & 0 & \cdots & z_{n}
\end{array}\right), \qquad\qquad
J_{n-1} = \left(\begin{array}{cccc}
1 & 1 & \cdots &  1\\
1 & 1 & \cdots & 1\\
\vdots & \vdots & \ddots & \vdots\\
1 & 1 & \cdots & 1
\end{array}\right),
\end{equation}
and $I_{n-1}$ is the identity matrix of size $n-1$.
\end{proposition}

Denote by ${\rm Tr}\, M$ the trace of a square matrix $M$.
Our results rely on the following observation.
\begin{lemma}
Let $M_{n-1}$ and $D_{n-1}$ be defined as in Proposition 2
and $i=\sqrt{-1}$. Then
\begin{equation} \label{m&d}
{\rm{Tr}}\left( e^{i t M_{n-1}} \right) - {\rm{Tr}}\left( e^{i t D_{n-1}} \right)
=\frac { i c_{n-1}}n {\rm{Tr}}\left( \tilde{J}_{n-1}\,e^{i t D_{n-1}} \right) ,
\end{equation}
where
\begin{eqnarray*}
c_{n-1} & = & c_{n-1}(t) = \int_0^t \exp \left(  i u \frac {\tilde{S}_{n-1}}n \right) \,du, 
 \\   
{\tilde S}_{n-1} & = & (z_1-z_2)+\cdots+(z_1-z_{n}),  \text{ and } \\ 
\tilde{J}_{n-1} & = &
\left(\begin{array}{cccc}
z_1-z_2 & 0 & \cdots &  0\\
0 & z_1-z_3 & \cdots & 0\\
\vdots & \vdots & \ddots & \vdots\\
0 & 0 & \cdots & z_1-z_{n}
\end{array}\right) J_{n-1}\,.  
\end{eqnarray*} 
\end{lemma}

\begin{proof}
It can be proved straightforwardly by Taylor expansions.
However, to derive the constant $c_{n-1}$ in a more natural way,
we first use the Duhamel formula 
$$
e^{(L_1+L_2)t} - e^{L_1 t}= \int_0^t e^{ L_1(t-\tau)} L_2 e^{(L_1+L_2)\tau}\, d\tau
$$
with $L_1= i D_{n-1}$ and $L_1+L_2= i M_{n-1}$.
In fact, $L_2=\frac {i}n \tilde J_{n-1}$.
Since ${\rm Tr}(e^{A+B})={\rm Tr} (e^A\,e^B)$ by ${\rm Tr} (AB) = {\rm Tr}(BA)$,
the left hand side of  (\ref{m&d}) equals
\begin{equation}\label{lhs}
\frac {i}n \int_0^t {\rm Tr} \left(\tilde J_{n-1} e^{ i t D_{n-1}}\, 
e^{ ( i u \tilde J_{n-1} )/n} \right)\, du.
\end{equation}
One can show by induction that 
$\displaystyle{\tilde J_{n-1}^k = \tilde S_{n-1}^{k-1} \tilde J_{n-1}, k\in \mathbb N}$.
After using this fact in the Taylor expansion of $\displaystyle{e^{ ( i u \tilde J_{n-1} )/n}}$,
the integrand within the integral of (\ref{lhs}) can be simplified to
\begin{equation}\label{trace}
{\rm Tr} \left(\tilde J_{n-1} e^{ i t D_{n-1}} \right)\, e^{ ( i u \tilde S_{n-1} )/n}. 
\end{equation}
This completes the proof.
\end{proof}

\begin{remark}
The $D_{n-1}$ appeared in \citep{cn} is a diagonal matrix with diagonal
entries $\{ z_1, \ldots, z_{n-1} \}$, while we choose a different one given in (\ref{d})
by the symmetry among the zeros. The $M_{n-1}$ in (\ref{m})
is modified accordingly.
\end{remark}
\begin{remark}
Observe that the terms in (\ref{trace}) can be expressed in a more symmetric way:
\begin{eqnarray}
\frac {\tilde S_{n-1} }{n} & = & z_1 - \frac {\sum_{j=1}^n z_j}n, \label{sn} \\
\frac 1n {\rm Tr} \left(\tilde J_{n-1} e^{ i t D_{n-1}} \right) & = &
z_1\,\left( \frac { \sum_{j=1}^n e^{i t z_j} }n \right) 
- \frac { \sum_{j=1}^n z_je^{i t z_j} }n .\label{t/n}
\end{eqnarray}
Suppose that $\{z_j\}_{j=1}^n$ satisfies the bound 
$\displaystyle{\max_{1 \le j \le n} | \Im z_j | \le C_0}$,
then it is easy get that 
\begin{eqnarray}
| c_{n-1} (t) | & \le & \frac {e^{2C_0 | t |} - 1 }{2C_0}, \label{c} \\ 
\left| \frac 1n {\rm Tr} \left(\tilde J_{n-1} e^{ i t D_{n-1}} \right) \right| & \le &
e^{C_0 | t |}  \left( | z_1 |  + \frac {\sum_{j=1}^n | z_j |}n \right). \label{abst/n}
\end{eqnarray}
When all zeros are real, $C_0=0$ and one simply gets $| c_{n-1} (t) |\le |t|$.
These are used in the proofs of three theorems.
\end{remark}

\section{Proofs and concluding remarks}

We now prove Theorem 1.
\begin{proof} 
Consider the case of $L_{n, 1} (f)$ and 
$L^{(1)}_{n, 1} (f)$ ($k=1$).
To prove the theorem in this case,
it suffices to show that 
$L_{n, 1} (f) - L^{(1)}_{n, 1} (f) \to 0$ a.s. as $n\to\infty$.
Write  $L_{n, 1} (f) - L^{(1)}_{n, 1} (f) = W_{n, 1} + W_{n, 2}$, where
\begin{eqnarray*}
W_{n, 1} & = & \frac {1}{a_n}   \left[ f\left( Z^{(n)}_1 \right) 
+ \cdots + f \left( Z^{(n)}_n \right) \right] - 
\frac {1}{a_{n-1}}   \left[ f\left( Z^{(n)}_2 \right) 
+ \cdots + f \left( Z^{(n)}_n \right) \right] \\
& = & \frac { f\left( Z^{(n)}_1 \right) }{a_n} + \frac{  a_{n-1} - a_n }{a_{n-1} a_n} \sum_{j=2}^n
f\left( Z^{(n)}_j \right),
\end{eqnarray*}
and 
$$
W_{n, 2} = \frac {1}{a_{n-1}}   \left[ f\left( Z^{(n)}_2 \right) 
+ \cdots + f \left( Z^{(n)}_n \right) \right]
- \frac {1}{a_{n-1}}   \left[ f\left( X^{n, 1}_1 \right) 
+ \cdots + f \left( X^{n, 1}_{n-1} \right) \right].
$$
Clearly
$$
| W_{n, 1} |\le \frac {\|f\|_\infty}{a_n} 
+ \frac{ n \,| a_{n-1} - a_n | \|f\|_\infty }{a_{n-1} a_n} \to 0 \quad \text{ a.s.}
$$
by (\ref{an1}) and (\ref{an2}) of {\bf A2}.
Next we apply Lemma 3 with $Z^{(n)}_j$ in place of $z_j, j=1, \ldots, n$, to obtain that
\begin{eqnarray*}
W_{n, 2} & = & \frac 1{a_{n-1}} \int \hat{ f }( t )
\left[ {\rm Tr } \left( e^{i t D_{n-1}} \right) - 
{\rm{Tr}}\left( e^{i t M_{n-1}} \right) \right]\, dt \\
& = & \frac 1{a_{n-1}} \int \hat{ f }( t ) \frac {c_{n-1}(t)}n {\rm{Tr}}
\left( \tilde{J}_{n-1}\,e^{i t D_{n-1}} \right) \,dt.
\end{eqnarray*}
Since we assume {\bf A1}, the estimates 
(\ref{c}) and (\ref{abst/n}) in Remark 4 can be used to yield that
\begin{equation}\label{wn2}
| W_{n, 2} | \le \frac 1{2C_0} \left( \int | {\hat f}(t) | e^{3C_0 | t |}\,dt \right)
\left( \frac { | Z^{(n)}_1 | }{a_{n-1}} + \frac {\sum_{j=1}^n | Z^{(n)}_j |}
{n a_{n-1}} \right). 
\end{equation}
By the regularity condition (\ref{f-bdd}) 
together with (\ref{an1}) and (\ref{an3}) of {\bf A2},
we conclude that $| W_{n, 2} | \to 0$ a.s. when $n \to \infty$.

For $k>1$ we decompose
$L_{n, 1} (f) - L_{n, 1}^{(k)} (f)
=L_{n, 1} (f) - L_{n, 1}^{(1)} (f) +\sum_{j=1}^{k-1} 
L_{n, 1}^{(j)} (f) - L_{n, 1}^{(j+1)} (f)$ and need to show that
each $ | L_{n, 1}^{(j+1)} (f) - L_{n, 1}^{(j)} (f) | \to 0 $ a.s.
We simply follow the same strategy as above. 
Two facts can be useful when estimating the difference between the traces of two matrices.
First one is the basic relation between roots and coefficients:
\begin{equation*} 
\frac{Z_1^{(n)}+\cdots +Z_n^{(n)}}{n} =\frac{X^{n, k}_1+\cdots+X^{n, k}_{n-k}}
{n-k}, k<n.
\end{equation*}
The second relation can be found in \citep{bs47} and \citep{en} :
$$
\frac{ | X^{n, k}_1|+\cdots+|X^{n, k}_{n-k}|}{n-k} \le \cdots
\le \frac{ | X^{n, 1}_1| + \cdots + | X^{n, 1}_{n-1}|}{n-1} \le
\frac{ | Z_1^{(n)}|+ \cdots + | Z_n^{(n)} |}{n}.
$$
The rest is easy and is omitted.
When demonstrating the weak convergence of
${\bar L}^{(k)}_{n, 1}$ from that of ${\bar L}_{n, 1}$,
the converging together lemma should be applied to complete the proof.
\end{proof}

The proofs of Theorem 2 and Theorem 3 are similar
to that of Theorem 1.
A mean value inequality and the boundedness of $| f' |$
can justify the transference of the scales
from $b_{n-k}$ to $b_{n-k-1}$.
When estimating the terms like (\ref{wn2}), 
the uniform condition (\ref{bn3}) in {\bf A3} and/or 
(\ref{anbn3}) in {\bf A4} would be helpful.
The rest parts are routine and thus omitted.

\begin{remark}
In the proofs of the theorems we deal with the strongest mode of convergence, namely
the almost sure convergence, of $L_{n, \ell} (f) - L_{n, \ell}^{(k)} (f) \to 0,
\ell = 1, 2, 3$.
If the original $L_{n, \ell} (f), \ell = 1, 3$, converges at a weaker mode,
an alternative argument using weaker estimates should be considered.
Consequently, it is conceivable that the assumptions weaker than
{\bf A1} and {\bf A2} (or {\bf A4}) might 
be sufficient for the theorems to hold, and 
the regularity conditions on the test functions used in the linear statistics
might also be relaxed.
\end{remark}

\begin{remark}
One can show that if the empirical measure associated with
the zeros of $p_n$, under appropriate scaling, obeys
a large deviations principle, so does the empirical measure associated with
the zeros of $p_n^{(k)}$ for each $k$.
This is treated elsewhere (\citep{2}).
\end{remark}

\bibliographystyle{amsplain}

\end{document}